\documentclass[a4paper, 11pt]{amsart}    
\usepackage{latexsym, amsmath, amsthm, amssymb, setspace, verbatim}
\usepackage[utf8]{inputenc} \usepackage[T1]{fontenc} \usepackage{lmodern}
\usepackage{hyperref, aliascnt}
\usepackage{enumitem}

\title[Embeddings into cubical shifts]{On embeddings of extensions of almost finite actions into cubical shifts}
\author{Emiel Lanckriet}
\address{Department of Computer Science, KU Leuven, Celestijnenlaan 200a\linebreak
\phantom{-}\hspace{2mm} box 2402, B-3001 Leuven, Belgium}
\email{emiel.lanckriet@kuleuven.be}
\author{Gábor Szabó}
\email{}
\address{Department of Mathematics, KU Leuven, Celestijnenlaan 200b, box 2400\linebreak
\phantom{-}\hspace{2mm} B-3001 Leuven, Belgium}
\email{gabor.szabo@kuleuven.be}
\makeatletter
\@namedef{subjclassname@2020}{%
  \textup{2020} Mathematics Subject Classification}
\makeatother
\subjclass[2020]{37B05}

\begin{document}

\newcommand{\into}[0]{\ensuremath{\hookrightarrow}}
\newcommand{\onto}[0]{\ensuremath{\twoheadrightarrow}}
\newcommand{\eps}[0]{\varepsilon}

\newcommand{\Pfin}[0]{\mathcal{P}_{\mathrm{fin}}}
\newcommand{\mdim}[0]{\ensuremath{\mathrm{mdim}}}
\newcommand{\mesh}[0]{\ensuremath{\mathrm{mesh}}}
\newcommand{\widim}[0]{\ensuremath{\mathrm{widim}}}

% theorems
\newtheorem{satz}{Satz}[section]		
\newaliascnt{corCT}{satz}
\newtheorem{cor}[corCT]{Corollary}
\aliascntresetthe{corCT}
\providecommand*{\corCTautorefname}{Corollary}
\newaliascnt{lemmaCT}{satz}
\newtheorem{lemma}[lemmaCT]{Lemma}
\aliascntresetthe{lemmaCT}
\providecommand*{\lemmaCTautorefname}{Lemma}
\newaliascnt{propCT}{satz}
\newtheorem{prop}[propCT]{Proposition}
\aliascntresetthe{propCT}
\providecommand*{\propCTautorefname}{Proposition}
\newaliascnt{theoremCT}{satz}
\newtheorem{theorem}[theoremCT]{Theorem}
\aliascntresetthe{theoremCT}
\providecommand*{\theoremCTautorefname}{Theorem}
\newtheorem*{theoreme}{Theorem}

\theoremstyle{definition}

\newaliascnt{conjectureCT}{satz}
\newtheorem{conjecture}[conjectureCT]{Conjecture}
\aliascntresetthe{conjectureCT}
\providecommand*{\conjectureCTautorefname}{Conjecture}
\newaliascnt{defiCT}{satz}
\newtheorem{defi}[defiCT]{Definition}
\aliascntresetthe{defiCT}
\providecommand*{\defiCTautorefname}{Definition}
\newaliascnt{remCT}{satz}
\newtheorem{rem}[remCT]{Remark}
\aliascntresetthe{remCT}
\providecommand*{\remCTautorefname}{Remark}
\newaliascnt{exampleCT}{satz}
\newtheorem{example}[exampleCT]{Example}
\aliascntresetthe{exampleCT}
\providecommand*{\exampleCTautorefname}{Example}

%%%%%%%%%%%%%%%%%%%%%%%%%%%%%%%%%%%%%%%%%%%%

\begin{abstract} 
For a countable amenable group $G$ and a fixed dimension $m\geq 1$, we investigate when it is possible to embed a $G$-space $X$ into the $m$-dimensional cubical shift $([0,1]^m)^G$.
We focus our attention on systems that arise as an extension of an almost finite $G$-action on a totally disconnected space $Y$, in the sense of Matui and Kerr.
We show that if such a $G$-space $X$ has mean dimension less than $m/2$, then $X$ embeds into the $(m+1)$-dimensional cubical shift.
If the distinguished factor $G$-space $Y$ is assumed to be a subshift of finite type, then this can be improved to an embedding into the $m$-dimensional cubical shift.
This result ought to be viewed as the generalization of a theorem by Gutman--Tsukamoto for $G=\mathbb Z$ to actions of all amenable groups, and represents the first result supporting the Lindenstrauss--Tsukamoto conjecture for actions of groups other than $G=\mathbb{Z}^k$.
\end{abstract}

\maketitle

%\setcounter{tocdepth}{1}
%\tableofcontents

%%%%%%%%%%%%%%%%%%%%%%%%%%%

\section*{Introduction}

It is a ubiquitous phenomenon in mathematics that if one deals with a category of objects with any kind of rich structure, there are often some natural distinguished examples that are \emph{large} enough to study conditions under which a general object embeds into the distinguished examples.
Depending on the precise context, the solution to such a problem can reveal an a priori surprising hierarchy present in the objects under consideration, or in the best case scenario give rise to new invariants that may have applications far beyond the original embedding problem.

In order to illustrate the historical importance of embedding problems, one need not look further than geometry and/or topology.
The Whitney embedding theorem, asserting that every $m$-dimensional Riemannian manifold embeds smoothly as a submanifold of $\mathbb R^{2m}$, was not only impactful for its statement, but introduced various concepts in its proof that remain fundamental in the area of differential geometry to this day.
The Menger--Nöbeling theorem, asserting that every compact metrizable space with covering dimension $m$ embeds continuously into the cube $[0,1]^{2m+1}$, not only generalizes this kind of phenomenon from the geometric context, but shows that covering dimension introduces a level of hierarchy among spaces having consequences beyond simply acting as a simple-minded obstruction to embeddings.

This short article aims to make progress on a similar embedding problem for topological dynamical systems, i.e., countably infinite discrete groups $G$ acting via homeomorphisms on compact metrizable spaces.
In this case the distinguished examples are given by so-called cubical shifts.
That is, given a natural number $m\geq 1$, we may consider $G$ acting on the space $([0,1]^m)^G$ by sending $g\in G$ to the homeomorphism $[(x_h)_{h\in G}\mapsto (x_{g^{-1}h})_{h\in G}]$; we refer to this as the \emph{$m$-dimensional cubical shift}.
A priori, it is not at all clear how to determine when a given action $\alpha: G\curvearrowright X$ embeds into such an example, not even how to determine that it does not.
Given that $X$ embeds into the Hilbert cube as a consequence of Urysohn--Tietze, say via $\iota: X\into [0,1]^{\mathbb N}$, it is a triviality to obtain the equivariant embedding into the analogous Hilbert cube shift $([0,1]^{\mathbb N})^G$ via $x\mapsto (\iota(\alpha_g(x)))_{g\in G}$.
Since $([0,1]^{\mathbb N})^G$ is actually homeomorphic to $([0,1]^{m})^G$, this begs the question how much of a hierarchy there really is between cubical shifts of different dimensions regarding the class of dynamical systems that embed into them.

The first really substantial embedding result was in the PhD thesis of Jaworski, who showed that every aperiodic homeomorphism on a finite-dimensional space (which we view as a free $\mathbb Z$-system) embeds equivariantly into $[0,1]^{\mathbb Z}$.
Later on this led to the question by Auslander, asking whether this holds for arbitrary aperiodic homeomorphisms on any space.
This problem remained open for over a decade before it was settled in the negative by the introduction of mean topological dimension, the ideas of which initially appeared in Gromov's work \cite{Gromov99} and were subsequently fleshed out by Lindenstrauss--Weiss \cite{LindenstraussWeiss00}.
Under the assumption that $G$ is amenable, every topological dynamical system $\alpha: G\curvearrowright X$ can be assigned its mean dimension $\mdim(X,\alpha)\in [0,\infty]$, which respects embeddings. (Although one should perhaps mention that mean dimension has since been extended to sofic groups \cite{Li13}, the methods in this paper reveal nothing new beyond the amenable case, hence we shall ignore sofic mean dimension here.)
In a nutshell, mean dimension is a dimensional analog of entropy and is designed to be useful for distinguishing systems of infinite topological entropy.
The conceptual difference between these notions can be summarized by the slogan that entropy measures the number of bits per second needed to describe points in a system, whereas mean dimension measures the number of real parameters per second.
From this intuitive perspective, it is not surprising that the mean dimension of every $m$-dimensional cubical shift is equal to $m$.
By the mere existence of free minimal actions with arbitrarily large mean dimension --- see \cite[§3]{LindenstraussWeiss00} and \cite{Krieger09} --- one gets plenty of examples that cannot embed into the $m$-dimensional cubical shift.
In a suprising twist at the time, Lindenstrauss in \cite{Lindenstrauss99} proved that (extensions of) minimal homeomorphisms with mean dimension less than $m/36$ do embed, however.
This has triggered the search for the optimal embedding result that can be seen as the dynamical generalization of the Menger--Nöbeling theorem.

Although the situation for completely general systems is rather subtle and unsolved even for $G=\mathbb Z$, there has been amazing progress for aperiodic or even minimal homeomorphisms.
Building on various substantial precursor results \cite{LindenstraussTsukamoto14, GutmanTsukamoto14, Gutman15, Gutman17}, the optimal embedding result was recently proved by Gutman--Tsukamoto \cite{GutmanTsukamoto20} for minimal homeomorphisms:\ Every minimal homeomorphism with mean dimension less than $m/2$ embeds into the $m$-dimensional cubical shift.
A generalization of this result for $\mathbb Z^k$-actions was successfully pursued in \cite{Gutman11, GutmanQiaoSzabo18, GutmanLindenstraussTsukamoto16, GutmanQiaoTsukamoto19}, the final approach of which involves extremely sophisticated tools from signal analysis to take advantage of the surrounding geometry for these groups.
As was noted in the introduction of \cite{GutmanTsukamoto20}, ``the generalization
to non-commutative groups seems to require substantially new ideas''.
Indeed there has been no progress on the embedding problem for dynamical systems over nonabelian groups to the best of the authors' knowledge, and this article aims to change that.
Our main result (\autoref{thm:embedding-result}+\autoref{cor:optimal-embedding}) asserts:

\begin{theoreme}
Suppose $G$ is a countable amenable group.
Let $\beta: G\curvearrowright Y$ be an almost finite action on a compact totally disconnected metrizable space.
Let $\alpha: G\curvearrowright X$ be an action on a compact metrizable space that arises as an extension of $\beta$.
Let $m \geq 1$ be a natural number and suppose that $\mdim(X, \alpha) < \frac{m}{2}$.
Then there exists an embedding of $G$-spaces $X\into ([0,1]^{m+1})^G$.
If $(Y,\beta)$ is assumed to be a subshift of finite type, then there exists an embedding of $G$-spaces $X\into ([0,1]^{m})^G$.
\end{theoreme}

In the context of the above theorem, we remark that the concept of almost finiteness for actions, introduced in \cite{Matui12, Kerr20} with a motivation towards $\mathrm{C}^*$-algebraic applications, is a kind of freeness property that is designed as a topological version of the Ornstein--Weiss lemma \cite{OrnsteinWeiss87} for free probability measure preserving actions.
Since it is by now known for a large class of groups that almost finiteness for $\beta$ follows if $\beta$ is assumed to be free (see \autoref{rem:almost-finite-actions}), our main result should be viewed as a generalization of Gutman--Tsukamoto's approach from \cite{GutmanTsukamoto14} to the setting of amenable groups.
This is indeed reflected not just in the similarity of the main result, but at the level of our proof.
More specifically, there are clear parallels between \autoref{lem:dense-eps-embeddings}, \autoref{thm:embedding-result} and \autoref{cor:optimal-embedding} on the one hand, and \cite[Proposition 3.1, Theorem 1.5, Corollary 1.8]{GutmanTsukamoto14} on the other hand.
In a nutshell, almost finiteness of $\beta$ in our proof acts as the correct substitute of the well-known clopen Rokhlin lemma for aperiodic homeomorphisms on the Cantor set.
We further point out that, to the best of our knowledge, this provides the first application of almost finiteness to prove a new result in topological dynamics that is entirely unrelated to questions about crossed product $\mathrm{C}^*$-algebras.

The problem whether the above result is true for all free actions $\alpha: G\curvearrowright X$, regardless of whether it admits well-behaved factor systems, remains open.
In light of the technical difficulties already present in the state-of-the-art for $\mathbb Z^k$, however, we expect this challenge to be rather difficult to tackle without ideas that go substantially beyond our present work.

%%%%%%%%%%%%%%%%%%%%%%%%%%%%%%%%%%%%%%%%%%%%%%%%%%%%%

\section{Preliminaries}

We start with some basic remarks on notation and terminology.

Throughout the article we fix a countable amenable group $G$.
We write $F\Subset G$ to mean that $F$ is a finite subset of $G$.
Given $K\Subset G$ and a constant $\delta>0$, we say that a non-empty set $F\Subset G$ is \emph{$(K,\delta)$-invariant}, if $|KF\setminus F|\leq\delta|F|$.
We will freely use the well-known characterization of amenability via the F{\o}lner criterion, i.e., $G$ is amenable precisely when every pair $(K,\delta)$ admits some $(K,\delta)$-invariant finite subset in $G$.
If $G$ is countable, we call a sequence $(F_n)_{n\in\mathbb N}$ with $F_n\Subset G$ a \emph{F{\o}lner sequence}, if for every pair $(K,\delta)$, there is some $n_0\in\mathbb N$ such that $F_n$ is $(K,\delta)$-invariant for all $n\geq n_0$.

The letters $X$ and $Y$ will always be reserved to denote compact metrizable spaces.
Under a \emph{topological dynamical system} (over $G$) or {\emph{$G$-space}} we understand a pair $(X,\alpha)$, where $X$ is a compact metrizable space and $\alpha: G\curvearrowright X$ is an action by homeomorphisms.
In some cases when there is no ambiguity on what action is considered on $X$, we sometimes just talk of the $G$-space $X$ to lighten notation.
An action $\alpha$ is called \emph{free} if for every point $x\in X$, its orbit map $[g\mapsto\alpha_g(x)]$ is injective.

Given another action $\beta: G\curvearrowright Y$, we say that a continuous map  $\phi: X\to Y$ is \emph{equivariant (w.r.t.\ $\alpha$ and $\beta$)}, if $\phi\circ\alpha_g=\beta_g\circ\phi$ for all $g\in G$, in which case we indicate this by writing $\phi: (X,\alpha)\to (Y,\beta)$.
Using the alternate arrow $\onto$ means that the map is surjective, whereas using $\into$ means that the map is injective, in which case we also speak of an embedding.
If we are given an equivariant surjective map $\pi: (X,\alpha)\onto (Y,\beta)$, then one calls $(Y,\beta)$ a \emph{factor} of $(X,\alpha)$ and refers to $\pi$ as the \emph{factor map}.
On the flip side, one says that $(X,\alpha)$ is an \emph{extension} of $(Y,\beta)$.

Of particular importance for this work is the example given by cubical shifts over a group $G$.
That is, given a natural number $m\geq 1$, the \emph{$m$-dimensional cubical shift} is the action $\sigma: G\curvearrowright ([0,1]^m)^G$ given by $\sigma_g\big( (x_h)_{h\in G} \big)=(x_{g^{-1}h})_{h\in G}$.

Let us now introduce the concepts underpinning this article, as well as some known results from the literature.

\subsection{Almost finiteness}

\begin{defi}
Let $\alpha: G\curvearrowright X$ be an action.
\begin{enumerate}[leftmargin=*,label=$\bullet$]
\item A {\it tower} is a pair $(V,S)$ consisting of a subset $V$ of $X$ and a finite subset $S$ of $G$ such that the sets $\alpha_s(V)$ for $s\in S$ are pairwise disjoint.
\item Given such a tower, the set $V$ is the {\it base} of the tower, the set $S$ is the {\it shape} of the tower, and the sets $\alpha_s(V)$ for $s\in S$ are the {\it levels} of the tower.
\item The tower $(V,S)$ is {\it open} if $V$ is open.
It is called {\it clopen} if $V$ is clopen.
\item A {\it castle} is a finite collection of towers $\{ (V_i , S_i) \}_{i\in I}$
such that for all $i,j\in I$ and $s\in S_i$, $t\in S_j$, we have that $\alpha_{s}(V_i)\cap\alpha_{t}(V_j)=\emptyset$ if $i\neq j$ or $s\neq t$.
\item The castle is {\it open} if each of the towers is open, and {\it clopen} if each of the towers is clopen.
\end{enumerate}
\end{defi}

The following definition originates in \cite[Definition 6.2]{Matui12} for principal ample groupoids, which was then adapted in \cite[Definition 8.2]{Kerr20} for actions of amenable groups on arbitrary spaces.
Although not trivially identical to the general version, the definition below is known to be an equivalent one in our setting due to \cite[Theorem 10.2]{Kerr20}. 

\begin{defi}
Let $\beta: G\curvearrowright Y$ be an action on a totally disconnected space.
We say that $\beta$ is \emph{almost finite}, if for every $K\Subset G$ and $\delta>0$, there exists a clopen castle $\{ (W_i , S_i) \}_{i\in I}$ such that $Y=\bigsqcup_{i\in I} \bigsqcup_{s\in S_i} \beta_s(W_i)$ and for every $i\in I$, the shape $S_i$ is $(K,\delta)$-invariant.
\end{defi}

\begin{rem} \label{rem:almost-finite-actions}
One of the possible ways to view almost finiteness is as a strong topological variant of the Ornstein--Weiss tower lemma \cite[Theorem 5]{OrnsteinWeiss87} that characterizes freeness of probability measure preserving actions in ergodic theory, which was recently strengthened in \cite{CJKMSTD18}.
Conjecturally, every free action $\beta: G\curvearrowright Y$ on a totally disconnected space is almost finite.\footnote{We note that the converse is not true for all groups $G$. In general one can only conclude from almost finiteness that the action is \emph{essentially free}, i.e., sets of the form $\{y\in Y\mid \beta_g(y)=y\}$ vanish under all $\beta$-invariant Borel probability measures. Examples of almost finite but non-free actions are found among generalized Odometers; see \cite{OrtegaScarparo20}.}
This is not so hard to see for $G=\mathbb Z$, as almost finiteness just boils down to the well-known clopen Rokhlin tower lemma for aperiodic homeomorphisms; see for example \cite[Proposition 3]{BezuglyiDooleyMedynets05}.
Although the general case is still open, the following partial results are by now known:
\begin{enumerate}[leftmargin=*,label=$\bullet$]
\item For any amenable group $G$, almost finite actions on the Cantor set are generic among all free minimal $G$-actions; see \cite[Theorem 4.2]{CJKMSTD18}.
\item The conjecture holds when $G$ has local subexponential growth, i.e., given any $F\Subset G$, one has $\lim_{n\to\infty}\frac{|F^{n+1}|}{|F^n|}=1$.
This was shown in \cite{KerrSzabo18} as a consequence of \cite{DownarowiczZhang17}.
\item Let $H\leq G$ be a normal subgroup so that the above conjecture holds for $H$-actions.
If $G/H$ is finite or cyclic, then the conjecture holds for all $G$-actions; see \cite{KerrNaryshkin21}.
In particular, the conjecture is verified for all elementary amenable groups.
\end{enumerate}
\end{rem}

\subsection{Mean dimension}

\begin{defi} \label{def:open-covers}
Given a finite open cover $\mathcal U$ of a topological space $X$, we define its \emph{order} as the minimal number $n\geq 0$ such that every point $x\in X$ is an element of at most $n+1$ members of $\mathcal U$.
If $X$ is equipped with a metric $d$, then $\mesh_d(\mathcal U)$ is defined as the maximal diameter of a member of $\mathcal U$.
Given a constant $\eps>0$, one defines 
\[
\widim_\eps(X,d) = \min\{ \operatorname{ord}(\mathcal U) \mid \mathcal U \text{ is an open cover with } \mesh_d(\mathcal U)\leq\eps\}.
\]
\end{defi}

Before we can define mean dimension, we recall the following technical result, which is a non-trivial consequence of the Ornstein--Weiss quasitiling machinery.

\begin{theorem}[see {\cite[Theorem 6.1]{LindenstraussWeiss00}}] \label{thm:subadditive-convergence}
Let $G$ be a countable amenable group.
Denote by $\Pfin(G)$ the set of all non-empty finite subsets of $G$.
Suppose we are given a function $\phi: \Pfin(G)\to [0,\infty)$ satisfying the following conditions:
\begin{enumerate}[leftmargin=*,label=$\bullet$]
	\item $\phi(F_1) \leq \phi(F_2)$ whenever $F_1 \subseteq F_2$;
	\item $\phi(Fg) = \phi(F)$ for all $F \Subset G$ and $g \in G$;
	\item $\phi(F_1 \cup F_2) \leq \phi(F_1) + \phi(F_2)$ for all $F_1, F_2\Subset G$.
\end{enumerate}
Then there exists $b\geq 0$ such that for every $\eps > 0$ there exists $K\Subset G$ and $\delta > 0$ such that $\big| b - \frac{\phi(F)}{|F|} \big| \leq \eps$ for every $(K, \delta)$-invariant set $F\Subset G$.
\end{theorem}

\begin{prop}[see {\cite[Proposition 10.4.1]{Coornaert}}]
Let $G$ be a countable amenable group and $\alpha: G\curvearrowright X$ a topological dynamical system.
For a compatible metric $d$ on $X$ and $F\Subset G$, define the metric $d^\alpha_F$ via 
\[
d^\alpha_F(x,y)=\max_{g\in F} \ d(\alpha_g(x),\alpha_g(y)).
\]
Let $\eps>0$ be a constant.
Then the map
\[
\Pfin(G)\ \ni \ F\mapsto \widim_\eps(X,d^\alpha_F)
\]
has the properties as required by \autoref{thm:subadditive-convergence}.
Consequently, if $F_n\Subset G$ is a F{\o}lner sequence, then the limit
\[
\mdim_\eps(X,\alpha,d)=\lim_{n\to\infty} |F_n|^{-1}\widim_\eps(X,d^\alpha_{F_n}) \ \in \ [0,\infty]
\]
exists and is independent of the choice of $(F_n)_n$.
\end{prop}

\begin{defi} \label{def:mdim}
Let $G$ be a countable amenable group and $\alpha: G\curvearrowright X$ a topological dynamical system.
The \emph{mean dimension} of $(X, \alpha)$ is defined as
\[
\mdim(X,\alpha) = \sup_{\eps>0} \ \mdim_\eps(X,\alpha,d) \ \in \ [0,\infty],
\]
where $d$ is some compatible metric on $X$.\footnote{This definition contains the implicit claim that this supremum does not depend on the chosen metric. This is not completely trivial, but it is well-known; see \cite[Theorem 10.4.2]{Coornaert}.}
In the cases where the choice of the action $\alpha$ is implicitly clear from context, we just write $\mdim(X)$.
\end{defi}

\begin{example}[see {\cite[Proposition 3.3]{LindenstraussWeiss00}}]
For every natural number $m\geq 1$, we can consider the $m$-dimensional cubical shift $\sigma: G\curvearrowright ([0,1]^m)^G$ as defined before.
Then $\mdim(([0,1]^m)^G)=m$.
\end{example}

\begin{rem}
It is an easy consequence of its definition that mean dimension respects inclusions.
That is, given an equivariant inclusion $X_1\into X_2$ of $G$-spaces, one has the inequality $\mdim(X_1)\leq\mdim(X_2)$.
In light of the above, it follows immediately that mean dimension provides an obstruction to the embeddibility of a $G$-space $X$ into the $m$-dimensional cubical shift.
\end{rem}

%%%%%%%%%%%%%%%%%%%%%%%%%%

\section{The embedding result}

\begin{defi}
Let $(X,d)$ be a compact metric space and $\eps>0$ a constant.
A continuous map $f: X\to Z$ into another topological space is called an \emph{$\eps$-embedding}, if $\operatorname{diam}(f^{-1}(z))<\eps$ for all $z\in Z$.
\end{defi}

The following lemma by Gutman-Tsukamoto plays the same role in our proof of the main result as it did in the proof of theirs.

\begin{lemma}[{\cite[Lemma 2.1]{GutmanTsukamoto14}}] \label{lem:widim}
Let $(X,d)$ be a compact metric space, $m\geq 1$ a natural number and $f_0 :X\to [0,1]^m$ a continuous map.
Suppose that the numbers $\delta,\eps>0$ satisfy the implication
\[
d(x,y)<\eps\quad\implies\quad\|f_0(x)-f_0(y)\|_\infty <\delta.
\]
If $\widim_\eps(X,d) < m/2$, then there exists an $\eps$-embedding $f: X\to [0,1]^m$ satisfying
\[
\|f-f_0\|_\infty:=\max_{x\in X} \|f(x)-f_0(x)\|_\infty <\delta.
\]
\end{lemma}

\begin{defi}
Let $(X,\alpha)$ be a topological dynamical system, $m\geq 1$ a natural number and $f: X\to [0,1]^m$ a continuous map.
We then define a continuous equivariant map $I_f: X\to ([0,1]^m)^G$ via $I_f(x)=(f(\alpha_g(x)))_{g\in G}$.
\end{defi}

\begin{lemma} \label{lem:dense-eps-embeddings}
Let $\beta: G\curvearrowright Y$ be an almost finite action on a compact totally disconnected space.
Let $\alpha: G\curvearrowright X$ be an action on a compact metrizable space that arises as an extension of $\beta$ via the factor map $\pi: (X,\alpha)\onto (Y,\beta)$.
Let $m \geq 1$ be a natural number and suppose that $\mdim(X, \alpha) < \frac{m}{2}$.
Choose a compatible metric $d$ on $X$.
Then for any $\eta > 0$ the set of functions
\[
A_{\eta} = \{ f \in \mathcal C(X,[0,1]^m) \mid I_f \times \pi \text{ is an $\eta$-embedding} \}
\]
is dense in $\mathcal C(X,[0,1]^m)$ with respect to $\|\cdot\|_\infty$.
\end{lemma}
\begin{proof}
Let $f_0: X\to [0,1]^m$ be a continuous map and let $\eta, \delta>0$.
We shall argue that there exists $f\in A_\eta$ with $\|f-f_0\|_\infty<\delta$.
Since $f_0$ is uniformly continuous, we can find some $0<\eps\leq\eta$ that fits into the implication
\[
d(x,y)<\eps \quad\implies\quad\|f_0(x)-f_0(y)\|_\infty <\delta.
\]
By assumption, we have $\mdim_\eps(X,\alpha,d)\leq\mdim(X,\alpha)<m/2$.
Since $\mdim_\eps(X,\alpha,d)$ arises as a limit in the sense of \autoref{thm:subadditive-convergence}, we can find a constant $\gamma>0$ and $K\Subset G$ such that for every $(K,\gamma)$-invariant set $S\Subset G$, we have $\widim_\eps(X,d^\alpha_S)<|S|m/2$.
Since we assumed $\beta$ to be almost finite, we may find a clopen castle $\{(W_i,S_i)\}_{i\in I}$ with $(K,\gamma)$-invariant shapes and $Y=\bigsqcup_{i\in I}\bigsqcup_{s\in S_i} \beta_s(W_i)$.
By defining the pullbacks $Z_i=\pi^{-1}(W_i)$ for $i\in I$, we obtain the clopen castle $\{(Z_i,S_i)\}_{i\in I}$ partitioning $X$.

Given $i\in I$, we have in particular that $\widim_\eps(X,d^\alpha_{S_i})\leq |S_i|m/2$.
Consider the continuous map $F^0_i: X\to [0,1]^{|S_i|m}\cong ([0,1]^m)^{S_i}$ given by $F_i^0(x)=(f_0(\alpha_s(x)))_{s\in S_i}$.
Note that by design, we have the implication
\[
d^\alpha_{S_i}(x,y)<\eps \quad\implies\quad \|F^0_i(x)-F^0_i(y)\|_\infty <\delta.
\]
Using \autoref{lem:widim}, we may choose a continuous $\eps$-embedding $F_i: X\to ([0,1]^m)^{S_i}$ with respect to the metric $d^{\alpha}_{S_i}$ such that $\|F_i-F_i^0\|_\infty<\delta$.
We now define the continuous function $f: X\to [0,1]^m$ as follows.
If $x\in X$ is a point, choose the unique index $i\in I$ and $s\in S_i$ with $x\in\alpha_s(Z_i)$, and set $f(x)=F_i(\alpha_s^{-1}(x))(s)$.
Since this assignment is clearly continuous on each clopen set belonging to a partition of $X$, $f$ is indeed a well-defined continuous map.
We claim that $\|f-f_0\|_\infty<\delta$ and $f\in A_\eta$.
The first of these properties holds because given $x\in X$ as above, we see that
\[
f(x)=F_i(\alpha_s^{-1}(x))(s) \approx_\delta F_i^0(\alpha_s^{-1}(x))(s)=f_0(\alpha_s(\alpha_s^{-1}(x)))=f_0(x).
\]
So let us argue $f\in A_\eta$.
Suppose $x,y\in X$ are two points such that $(I_f\times\pi)(x)=(I_f\times\pi)(y)$.
Then certainly $\pi(x)=\pi(y)$.
Since the clopen partition $X=\bigsqcup_{i\in I}\bigsqcup_{s\in S_i} \alpha_s(Z_i)$ is the pullback from a clopen partition of $Y$, we see that there is a unique index $i\in I$ and $s\in S_i$ with $x,y\in\alpha_s(Z_i)$.
Since we have $I_f(x)=I_f(y)$, or in other words $f(\alpha_g(x))=f(\alpha_g(y))$ for all $g\in G$, it follows for all $t\in S_i$ that $\alpha_{ts^{-1}}(x), \alpha_{ts^{-1}}(y)\in \alpha_t(Z_i)$, so
\[
F_i(\alpha^{-1}_s(x))(t)= f(\alpha_{ts^{-1}}(x))=f(\alpha_{ts^{-1}}(y))=F_i(\alpha^{-1}_s(x))(t).
\]
Since $t\in S_i$ is arbitrary, it follows that $F_i(\alpha_s^{-1}(x))=F_i(\alpha_s^{-1}(y))$.
Since $F_i$ was an $\eps$-embedding with respect to the metric $d^\alpha_{S_i}$ and $s\in S_i$, we may finally conclude $d(x,y)<\eps\leq\eta$.
This finishes the proof.
\end{proof}

\begin{theorem} \label{thm:embedding-result}
Let $\beta: G\curvearrowright Y$ be an almost finite action on a compact totally disconnected metrizable space.
Let $\alpha: G\curvearrowright X$ be an action on a compact metrizable space that arises as an extension of $\beta$ via the factor map $\pi: (X,\alpha)\onto (Y,\beta)$.
Let $m \geq 1$ be a natural number and suppose that $\mdim(X, \alpha) < \frac{m}{2}$.
Then the set of functions $f\in\mathcal C(X,[0,1]^m)$ for which
\[
I_f\times\pi: (X,\alpha)\to \big( ([0,1]^m)^G\times Y, \sigma\times\beta \big)
\]
is an embedding, is dense with respect to $\|\cdot\|_\infty$.
Consequently, there exists an embedding of $G$-spaces $X\into ([0,1]^{m+1})^G$.
\end{theorem}
\begin{proof}
Let us first explain the last sentence of the claim.
Since $Y$ is totally disconnected, it can be embedded into $[0,1]$, say via a continuous map $\psi$.
This implies that $\bar{\psi}: Y\to [0,1]^G$ given by $y\mapsto(\psi(\beta_g(y)))_{g\in G}$ is an equivariant embedding.
So assuming the rest of the claim holds, we obtain a chain of embeddings of $G$-spaces
\[
X\stackrel{I_f\times\pi}{\longrightarrow} ([0,1]^m)^G\times Y \stackrel{\operatorname{id}\times\bar{\psi}}{\longrightarrow} ([0,1]^{m})^G\times ([0,1])^G \cong ([0,1]^{m+1})^G.
\]

If we adopt the notation from \autoref{lem:dense-eps-embeddings}, it is clear that the set of functions in question is equal to the intersection $\bigcap_{n\geq 1} A_{1/n}$.
In light of the fact that $\mathbb C(X,[0,1]^m)$ a closed subset of the Banach space $\mathcal C(X,\mathbb R^m)$ with respect to $\|\cdot\|_\infty$, the claim follows immediately from the Baire category theorem if we show that the sets $A_\eta$ are open for all $\eta>0$.

So let us briefly argue that this is the case.
Recall that we have chosen a compatible metric $d$ on $X$.
Let $f\in A_\eta$.
Given an infinite tuple $(c_g)_{g\in G}$ of strictly positive numbers with $\sum_{g\in G} c_g=1$, we define the constant $\delta$ via
\[
2\delta = \inf\Big\{ \sum_{g\in G} c_g\|f(\alpha_g(x))-f(\alpha_g(y)) \|_\infty \Big| x,y\in X, \pi(x)=\pi(y), d(x,y)\geq\eta \Big\}.
\]
Keep in mind that the assignment 
\[
\big( (z^{(1)}_g)_{g\in G}, (z^{(2)}_g)_{g\in G} \big) \mapsto\sum_{g\in G} c_g\|z^{(1)}_g-z^{(2)}_g\|_\infty
\]
defines a compatible metric on $([0,1]^m)^G$.
Since $I_f$ is continuous, $I_f\times\pi$ is an $\eta$-embedding and $X$ is compact, it follows that $\delta>0$.
We claim that the open $\delta$-ball around $f$ is contained in $A_\eta$.
Indeed, let $f_0\in\mathcal C(X,[0,1]^m)$ with $\|f-f_0\|<\delta$.
Suppose that $x,y\in X$ satisfy $(I_{f_0}\times\pi)(x)=(I_{f_0}\times\pi)(y)$.
Then $\pi(x)=\pi(y)$ and it follows from the triangle inequality that 
\[
\sum_{g\in G} c_g\|f(\alpha_g(x))-f(\alpha_g(y))\|_\infty < \sum_{g\in G} c_g(2\delta+\|f_0(\alpha_g(x))-f_0(\alpha_g(y))\|_\infty) = 2\delta.
\]
By the definition of $\delta$, it follows that $d(x,y)<\eta$.
Since $x$ and $y$ were arbitrary, we conclude $f_0\in A_\eta$ and the proof is finished.
\end{proof}

We also record an improved version of the embedding result, which is an immediate consequence of the above if we assume more about the system $(Y,\beta)$.

\begin{cor} \label{cor:optimal-embedding}
Let $\beta: G\curvearrowright Y$ be an almost finite action on a compact totally disconnected metrizable space.
Suppose that $\beta$ is a subshift of finite type, i.e., there exists some natural number $\ell\geq 2$ and an embedding $Y\into\{1,\dots,\ell\}^G$ of $G$-spaces.
Let $\alpha: G\curvearrowright X$ be an action on a compact metrizable space that arises as an extension of $\beta$.
Let $m \geq 1$ be a natural number and suppose that $\mdim(X, \alpha) < \frac{m}{2}$.
Then there exists an embedding of $G$-spaces $X\into ([0,1]^{m})^G$.
\end{cor}
\begin{proof}
Find some embedding $\varphi: [0,1]\times\{1,\dots,\ell\}\into [0,1]$, which gives rise to an equivariant embedding 
\[
\bar{\varphi}: [0,1]^G\times\{1,\dots,\ell\}^G \cong \big([0,1]\times\{1,\dots,\ell\}\big)^G \into [0,1]^G
\]
by applying $\varphi$ componentwise.
This allows us to proceed exactly as in the last part of \autoref{thm:embedding-result}, except that we may appeal to the embedding
\[
\begin{array}{ccl}
([0,1]^m)^G\times Y &\into& ([0,1]^m)^G\times\{1,\dots,\ell\}^G \\
&\cong& ([0,1]^{m-1})^G\times \big([0,1]\times\{1,\dots,\ell\}\big)^G \\
&\stackrel{\operatorname{id}\times\bar{\varphi}}{\longrightarrow}& ([0,1]^{m-1})^G\times ([0,1])^G\ \cong \ ([0,1]^{m})^G.
\end{array}
\]
\end{proof}

\begin{rem}
In light of the fact that almost finiteness is a concept that can be defined for actions on arbitrary spaces, one might wonder how far the main result of this note can be generalized.
Suppose $\gamma: G\curvearrowright Z$ is an almost finite action on a not necessarily disconnected space.
It is then well-known that $\gamma$ has the small boundary property and therefore also $\mdim(Z,\gamma)=0$; see \cite[Theorem 5.6]{KerrSzabo18} and \cite[Theorem 5.4]{LindenstraussWeiss00}.\footnote{Note that a priori, this reference in \cite{KerrSzabo18} assumes freeness of the action. However, the statement involves an ``if and only if'' statement where we only need the ``if'' part, which does not need freeness of the involved action in any way.} 
Can one prove directly that $(Z,\gamma)$ embeds into the 1-dimensional cubical shift?
If so, is the statement of \autoref{thm:embedding-result} true if we replace $\beta: G\curvearrowright Y$ by $\gamma: G\curvearrowright Z$?

Although this would seem plausible, the proof does by no means generalize in any obvious way to this more general case.
The first named author has proved a partial result in this direction in his master thesis \cite{Lanckriet21}, namely under the assumption that $Z$ has finite covering dimension $d$.
In that case, a version of \autoref{thm:embedding-result} is true, where the conclusion is weakened to obtain an embedding into the $(m(d+2)+1)$-dimensional cubical shift.
Since this dimensional upper bound is far from what we expect to be optimal, and since it does not actually recover \autoref{thm:embedding-result} as a special case, we decided not to include this generalized approach in this note.
\end{rem}

%%%%%%%%%%%%%%%%%%%%%%%%%%%%%%%%%%%%%%%%%%%%%%%%%%%%%%%%%%%%%%%%%%

\textbf{Acknowledgements.} 
The second named author has been supported by research project C14/19/088 funded by the research council of KU Leuven, and the project G085020N funded by the Research Foundation Flanders (FWO).

%%%%%%%%%%%%%%%%%%%%%%%%%%%%%%%%%%%%%%%%%%%%%%%%%%%%%%%%%%%%%%%%%%

\end{document}